\documentclass[a4paper, 12pt]{article}

\usepackage[T1]{fontenc} 
\usepackage[utf8]{inputenc} 
\usepackage{amsmath} 
\usepackage{amsfonts} 
\usepackage{amssymb} 
\usepackage{amsthm} 
\usepackage{mathrsfs}
\usepackage{bbm}

\usepackage{tikz}

\usepackage{hyperref}

\newcommand{\N}{\mathbb{N}} 
\newcommand{\R}{\mathbb{R}} 

\newcommand{\Prob}{\mathbb P}
\newcommand{\E}{\mathbb E}
\newcommand{\V}{\mathbb V}

\newcommand{\dd}{\,\mathrm{d}}
\newcommand{\id}{\overset{\mathrm{d}} {=}}
\DeclareMathOperator{\cov}{cov}
\DeclareMathOperator{\Span}{span}

\theoremstyle{plain} 
\newtheorem{theorem}{Theorem} 
\newtheorem{lemma}[theorem]{Lemma} 
\newtheorem{corollary}[theorem]{Corollary}

\theoremstyle{definition} 

\theoremstyle{remark} 

\begin{document}
\date{\today}
\title{Persistence probabilities of mixed FBM and other mixed processes}
\author{Frank Aurzada\footnote{Technical University of Darmstadt, Schloßgartenstraße 7, 64289 Darmstadt, Germany. E-mail: aurzada@mathematik.tu-darmstadt.de, kilian@mathematik.tu-darmstadt.de} \and Martin Kilian\footnotemark[1] \and Ercan Sönmez\footnote{University of Klagenfurt, Universitätsstraße 65-67, 9020 Klagenfurt, Austria. E-mail: ercan.soenmez@aau.at}}
\maketitle

\begin{abstract}
We consider the sum of two self-similar centred Gaussian processes with different self-similarity indices. Under non-negativity assumptions of covariance functions and some further minor conditions, we show that the asymptotic behaviour of the persistence probability of the sum is the same as for the single process with the greater self-similarity index. 

In particular, this covers the mixed fractional Brownian motion introduced in \cite{Cheridito2001} and shows that the corresponding persistence probability decays asymptotically polynomially with persistence exponent $1-\max(1/2,H),$ where $H$ is the Hurst parameter of the underlying fractional Brownian motion. 
\end{abstract}


\bigskip


\section{Introduction}
Studying the so-called persistence or survival probabilities of stochastic processes is a classical issue in probability theory. For real-valued processes, persistence in particular concerns the event of staying on a half-line for an untypically long time, and one is interested in the asymptotic behaviour of the probability of this event. For many processes $X=(X_t)_{t \ge 0}$ of interest, the persistence probability decays asymptotically polynomially to zero, i.e.
\begin{equation}\label{eq:pppolynomial}
\Prob{\left(\sup_{t \in [0,T]} X_t \le 1\right)} = T^{-\theta +o(1)},\qquad T\to\infty,
\end{equation}
for some constant $\theta \in(0,\infty),$ and the aim is to determine the so-called persistence exponent $\theta$ of $X$. This persistence exponent is of special interest for theoretical physicists as it serves as a simple measure of how fast a complicated physical system returns from a disordered initial condition to its stationary state. The analysis of persistence probabilities of various processes has received considerable attention in recent years, both in theoretical physics and mathematics. For an overview of relevant processes and known results from the theoretical physics point of view, we recommend the surveys \cite{Majumdar2013} and \cite{Majumdar1999} as well as the monographs \cite{Metzler2014} and \cite{Redner2001}, while we refer to \cite{AurzadaSimon2015} for a survey of the mathematics literature in this context. 

Recall that fractional Brownian motion (FBM) $B^H$ with Hurst parameter $H \in (0,1)$ is the unique normalized centred Gaussian process which is $H$-self-similar and has stationary increments. Molchan has shown already in 1999 that $B^H$ satisfies \eqref{eq:pppolynomial} for the persistence exponent $\theta=1-H,$ see \cite{Molchan1999}, but still up to now, persistence probabilities of FBM and related processes have been studied extensively in theoretical physics and mathematics. For instance, we refer to \cite{MolchanKhoklov2004} and \cite{Molchan2017} where the Hausdorff dimension of Lagrangian regular points for the inviscid Burgers' equation with FBM initial velocity is related to the two-sided persistence probabilities of integrated FBM, motivated by \cite{She1992} and \cite{Sinai1992P}. 

The main purpose of this paper is to show that the persistence probability of mixed FBM, as introduced in \cite{Cheridito2001}, asymptotically behaves indeed as in \eqref{eq:pppolynomial} and to determine the corresponding persistence exponent. Mixed FBM $M^{H,\alpha}$ is defined as
\begin{equation}\label{eq:mixedfbm}
M_t^{H,\alpha}:=B_t^{1/2} + \alpha B_t^H, \qquad t \ge 0,
\end{equation}
where $\alpha \in \R \setminus \{0\}, \ H \in (0,1), \ B^H$ is an FBM with Hurst parameter $H$ and $B^{1/2}$ is an independent Brownian motion. Note that this process still has stationary increments, but is {\it not self-similar} itself. 

In fact, we will derive the desired persistence result for a more general class of sums of self-similar centred Gaussian processes with different self-similarity indices, covering not only the mixed FBM $M^{H,\alpha},$ but also e.g.~the case of completely correlated mixed FBM introduced in \cite{Dufitinema2021}. Note that the latter process neither is self-similar nor has stationary increments. Thus, our result contributes to the amount of rather rare persistence results for stochastic processes violating both the properties of self-similarity and stationary increments. Self-similarity is a valuable property in the context of persistence as in this case, one is able to apply the so-called Lamperti transformation to get a stationary process and concerning persistence, many powerful tools are available for the class of stationary centred Gaussian processes, see \cite{DemboMukherjee2015}, \cite{Feldheim2015}, \cite{DemboMukherjee2017}, \cite{Basu2018}, \cite{Feldheim2021}, \cite{AurzadaMukherjee2020} and \cite{FeldheimMukherjee2021}. In particular, combined with non-negative (and non-degenerate) covariances, self-similarity always guarantees the existence of the persistence exponent. In the case that self-similarity is not available, the property of stationary increments turned out to be appropriate as another property that can be used to prove the existence of the persistence exponent, see \cite{AGPP2018}. Besides, one could derive persistence results even outside of the Gaussian setting if one assumes both self-similarity and stationary increments, see \cite{AurzadaMoench2019} and \cite{Moench2021}. 

The outline of the paper is as follows. In Section \ref{sec:result}, we will introduce the class of mixed self-similar processes that are suitable for our purposes and present our main result that for these processes, the persistence probability decays asymptotically polynomially with the persistence exponent of the self-similar process with the greater self-similarity index. In Section \ref{sec:corollaries}, we will then use this result to derive persistence results for the (completely correlated) mixed FBM and other explicit mixed processes of interest. Finally, in Section \ref{sec:proof}, we will prove the main result.

\section{Main result}\label{sec:result}
Recall that for $H>0,$ a stochastic process $(X_t)_{t \ge 0}$ is called $H$-self-similar if $(X_{ct}) \id (c^H X_t)$. We consider the sum of two self-similar centred Gaussian processes with different self-similarity indices, i.e.~$X^H+Y^K,$ where $X^H$ is an $H$-self-similar centred Gaussian process, $Y^K$ is a $K$-self-similar centred Gaussian process and $K<H$. The main result of this paper, which is given in the following theorem, states that under the assumption that $X^H$ and $X^H+Y^K$ have non-negative covariance functions, respectively, and that some minor conditions hold as specified below, the persistence probability of $X^H+Y^K$ has -- up to terms of lower order -- the same asymptotic behaviour as the persistence probability of $X^H$. 
\begin{theorem} \label{thm:thetheorem}
For $0 < K < H,$ let $X^H$ and $Y^K$ be self-similar centred Gaussian processes with a.s.~càdlàg sample paths and self-similarity indices $H$ and $K,$ respectively. Let us assume that the covariance functions of the processes $X^H$ and $X^H+Y^K$ are non-negative, respectively, and that the autocovariance function $\tau \mapsto \cov(X_1^H, e^{-\tau H} X_{e^\tau}^H)$ of the Lamperti transform of $X^H$ is continuous, integrable and not the zero function. Further let
\begin{equation*}
\theta_X:=-\lim_{T \to \infty} \frac{\log \Prob{\left( \sup_{t\in[0,T]} X^H_t \leq 1 \right)}} {\log T}
\end{equation*}
be the persistence exponent of $X^H$. Then
\begin{equation*}
\Prob{\left( \sup_{t\in[0,T]} X^H_t + Y^K_t \leq 1 \right)} = T^{-\theta_X +o(1)}, \qquad T \to \infty.
\end{equation*}
\end{theorem}
The result remains true if one replaces $Y^K$ by a finite sum of self-similar centred Gaussian processes $Y^{K_i}$ with self-similarity indices $K_i < H,$ as our proof in Section \ref{sec:proof} can be easily adapted to this setting. From a mathematics point of view, it would be an interesting open problem to study the persistence probability when $Y^K$ is replaced by a (well-defined) infinite sum of $Y^{K_i},$ where e.g.~$\sup_i K_i = H,$ as in this case, our proof technique does not work anymore. 

As already mentioned in the introduction, the assumption of non-negative covariances of $X^H$ (together with the self-similarity and the assumption on the Lamperti transform) guarantees the existence of $\theta_X$. Note that for the mixed process $X^H+Y^K$ on the contrary, the condition of non-negative covariances does not yield the existence of a persistence exponent a priori, as the mixed process is not self-similar anymore. Further note that we do not need any direct assumption on the covariance function of $Y^K$ or on the correlation of $X^H$ and $Y^K$. Thus, in particular, $X^H$ and $Y^K$ do not need to be independent and a persistence exponent of $Y^K$ does not necessarily have to exist.

\section{Mixed FBM and further corollaries}\label{sec:corollaries}
\paragraph{Mixed FBM.}
Let us now come back to the case of mixed fractional Brownian motion, which we defined in \eqref{eq:mixedfbm}. Note that this is a special case of the so-called fractional mixed fractional Brownian motion, which covers all linear combinations of independent fractional Brownian motions with different Hurst parameters, see \cite{ElNouty2003} and \cite{Miao2008}. Recall that fractional Brownian motion (FBM) $B^H$ has the covariance function $(s,t) \mapsto \frac {1} {2} \left(t^{2H} + s^{2H} - |t-s|^{2H}\right),$ which is non-negative. Due to the independence of the underlying fractional Brownian motions, this directly implies also the non-negativity of the covariance function of the (fractional) mixed FBM. Note that the continuous and integrable function $\tau \mapsto \frac {1} {2} \left(e^{H\tau} + e^{-H \tau} - |e^{\tau/2}-e^{-\tau/2}|^{2H}\right)$ is the autocovariance function of the Lamperti transform of $B^H$. Regarding persistence, it was shown in \cite{Molchan1999} that
\begin{equation}\label{eq:molchanpersistence}
\Prob{\left( \sup_{t\in[0,T]} B^H_t \leq 1 \right)}=T^{-(1-H)+o(1)}, \qquad T \to \infty.
\end{equation}
This yields the following corollary of Theorem \ref{thm:thetheorem} for the (fractional) mixed FBM. 
\begin{corollary}
For $0 < K < H < 1,$ let $B^H$ and $B^K$ be independent FBMs with Hurst parameters $H$ and $K,$ respectively, and $a,b \in \R$ with $ab \neq 0$. Then
\begin{equation*}
\Prob{\left( \sup_{t\in[0,T]} a B^H_t +  b B^K_t \leq 1 \right)} = T^{-(1-H) +o(1)}, \qquad T \to \infty.
\end{equation*}
In particular, for the mixed FBM as defined in \eqref{eq:mixedfbm}, we have
\begin{equation*}
\Prob{\left( \sup_{t\in[0,T]} M^{H,\alpha}_t \leq 1 \right)} = T^{-(1-\max\{\frac {1} {2},H\}) +o(1)}, \qquad T \to \infty.
\end{equation*} 
\end{corollary}
Note that the local behaviour of fractional mixed FBM is completely different: In \cite{VanZanten2007}, it was shown that $a B^H + b B^K$ is locally equivalent to $b B^K$ if and only if $H-K>1/4$.

\paragraph{Completely correlated mixed FBM.}
As mentioned in the introduction, Theorem \ref{thm:thetheorem} also covers the case of completely correlated mixed FBM. Under this term, it was introduced recently in \cite{Dufitinema2021}, while the process itself had already been studied as the driving process of an SDE in \cite[Section~3.2.3]{Mishura2008}. The definition is as follows. Let $B^H$ be an FBM with Hurst parameter $H \in (0,1)$. Then, there exists a Brownian motion $W$ such that
\begin{equation*}
B_t^H=\int_0^t K_H(t,s) \dd W_s, \qquad t \ge 0,
\end{equation*}
where $K_H$ is the so-called Molchan-Golosov kernel, see \cite[Section~5.1.3]{Nualart2006} and \eqref{eq:kernel1} and \eqref{eq:kernel2} below. Completely correlated mixed FBM (ccmFBM) $X^{H,a,b}$ is given by
\begin{equation}\label{eq:defCCMFBM}
X_t^{H,a,b}:=a W_t + b B_t^H, \qquad t \ge 0,
\end{equation}
where $a,b \in \R$ with $ab \neq 0$. Similarly to the fractional mixed FBM, as $K_{1/2} \equiv 1$ (see \eqref{eq:kernel2}), one can generalize $X^{H,a,b}$ to linear combinations $a B^H + b B^K$ of fractional Brownian motions generated by the same Brownian motion $W$ via the Molchan-Golosov kernels $K_H$ and $K_K$ with different Hurst parameters $H$ and $K,$ which were discussed recently in \cite{NualartSoenmez2021} and which we want to call fractional ccmFBM. Using the Itô-isometry, the fractional ccmFBM has the covariance function
\begin{align}
(s,t) \mapsto \ &a^2 \E[B_s^H B_t^H] + b^2 \E[B_s^K B_t^K] \notag
\\&+ ab \int_0^{s \wedge t} (K_H(t,u) K_K(s,u) + K_H(s,u) K_K(t,u)) \dd u. \label{eq:covarianceCCMFBM}
\end{align}
Set $C(H):=\sqrt{\frac {2H \,\Gamma{\left(\frac {3} {2}-H\right)} \,\Gamma{\left(H+\frac {1} {2}\right)}} {\Gamma{\left(2-2H\right)}}}$. Then, for $H>1/2$ and $0<s<t,$ we have
\begin{equation}\label{eq:kernel1}
K_H(t,s)=\frac {C(H)} {\Gamma{\left(H-\frac {1} {2}\right)}} \,s^{\frac {1} {2}-H} \int_s^t u^{H-\frac {1} {2}} (u-s)^{H-\frac {3} {2}} \dd u \ge 0,
\end{equation}
whereas for $H \le 1/2$ and $0<s<t,$ it holds
\begin{align}
&K_H(t,s) \notag
\\&=\frac {C(H)} {\Gamma{\left(H+\frac {1} {2}\right)}} \left((\frac {t^2} {s} - t)^{H-\frac {1} {2}} + (\frac {1} {2} - H) s^{\frac {1} {2} - H} \int_s^t u^{H - \frac {3} {2}} (u-s)^{H - \frac {1} {2}} \dd u\right) \ge 0. \label{eq:kernel2}
\end{align}
Thus, the covariance function of the (fractional) ccmFBM is non-negative, if $ab>0,$ and Theorem \ref{thm:thetheorem} together with \eqref{eq:molchanpersistence} gives the following corollary. 
\begin{corollary}
For $0 < K < H < 1$ and a Brownian motion $W,$ define $B_t^H:=\int_0^t K_H(t,s) \dd W_s$ and $B_t^K:=\int_0^t K_K(t,s) \dd W_s$. Further let $a,b \in \R$ with $ab>0$. Then
\begin{equation*}
\Prob{\left( \sup_{t\in[0,T]} a B_t^H +  b B^K_t \leq 1 \right)} = T^{-(1-H) +o(1)}, \qquad T \to \infty.
\end{equation*}
In particular, for the ccmFBM as defined in \eqref{eq:defCCMFBM}, we have
\begin{equation*}
\Prob{\left( \sup_{t\in[0,T]} X_t^{H,a,b} \leq 1 \right)} = T^{-(1-\max\{\frac {1} {2},H\}) +o(1)}, \qquad T \to \infty.
\end{equation*}
\end{corollary}
Further important self-similar centred Gaussian processes are integrated FBM and fractionally integrated Brownian motion, also called Riemann-Liouville process. Similarly to mixed FBM, one can define mixed integrated FBM and mixed Riemann-Liouville processes. 

\paragraph{Mixed integrated FBM.}
Let us first consider the case of integrated FBM. For $H \in (0,1),$ let $B^H$ be an FBM. Integrated FBM $I^H$ is then defined as
\begin{equation*}
I_t^H:=\int_0^t B_s^H \dd s, \qquad t \ge 0.
\end{equation*}
As $B^H$ is $H$-self-similar and has a non-negative covariance function, the so-defined $I^H$ is $(1+H)$-self-similar and has again a non-negative covariance function. The autocovariance function of the Lamperti transform of $I^H$ can be found in \cite[Lemma~2]{Molchan2008} and one easily sees that this is indeed a continuous and integrable function. As already mentioned in the context of Theorem \ref{thm:thetheorem}, this guarantees the asymptotic behaviour of the persistence probability of $I^H$ as in \eqref{eq:pppolynomial} with some persistence exponent $\theta_I(H) \in (0,\infty)$. However, the value $\theta_I(H)$ is unknown unless $H=1/2$: For integrated Brownian motion $I^{1/2},$ one could show using Markov techniques that $\theta_I(1/2)=1/4$ (see \cite{Goldman1971}, \cite{Sinai1992RW}, and \cite{IsozakiWatanabe1994}). For the general case, Molchan and Khoklov stated in \cite{MolchanKhoklov2004} the conjecture that $\theta_I(H)=H (1-H)$. In \cite{AurzadaKilian2020}, it was shown that $\theta_I$ is a continuous function and asymptotically equivalent to the conjectured $H (1-H)$ for $H \to 0$ and $H \to 1$. 

Therefore, Theorem \ref{thm:thetheorem} yields the following corollary for mixed integrated FBM. 
\begin{corollary}
For $0 < K < H < 1,$ let $B^H$ and $B^K$ be independent FBMs with Hurst parameters $H$ and $K,$ respectively, and $a,b \in \R$ with $ab \neq 0$. Let $I_t^H=\int_0^t B_s^H \dd s$ and $I_t^K=\int_0^t B_s^K \dd s$. Further let $\theta_I: (0,1) \to (0,\infty)$ denote the persistence exponent of integrated FBM depending on the Hurst parameter. Then
\begin{equation*}
\Prob{\left( \sup_{t\in[0,T]} a I^H_t +  b I^K_t \leq 1 \right)} = T^{-\theta_I(H) +o(1)}, \qquad T \to \infty.
\end{equation*}
\end{corollary}
Of course, the same result also holds for the integral of (fractional) ccmFBM, again in the case $ab > 0,$ as the only difference in verifying the assumptions of Theorem \ref{thm:thetheorem} is that the covariance function of the mixed process has additional summands. But these are given as the double integral of the additional summands in \eqref{eq:covarianceCCMFBM}, which is again non-negative if $ab>0$. 

\paragraph{Mixed Riemann-Liouville processes.}
As a last example, we want to consider mixed Riemann-Liouville processes, which were introduced in \cite[Section~8]{Cai2016}. For a Brownian motion $W$ and $H>0,$ define
\begin{equation*}
R_t^H:=\int_0^t (t-s)^{H-\frac {1} {2}} \dd W_s,\qquad t \ge 0,
\end{equation*}
to be the Riemann-Liouville fractional integral of $W$. Note that $R^{n+1/2}/n!$ for $n \in \N$ is simply the $n$-times integrated Brownian motion $W$. Thus, the Riemann-Liouville process $R^H$ for $H>0$ is a fractionally integrated Brownian motion. Further note that for $H \in (0,1),$ the process $R^H$ is closely related to FBM $B^H$ via the Mandelbrot-van Ness integral representation, which states that
\begin{equation*}
R_t^H+M_t^H:=\int_0^t (t-s)^{H- \frac {1} {2}} \dd W_s + \int_{-\infty}^0 (t-s)^{H-\frac {1} {2}} - (-s)^{H-\frac {1} {2}} \dd W_s
\end{equation*}
is an independent decomposition of FBM (with a non-normalized variance), see e.g.~\cite[Theorem~1.3.1]{Mishura2008}. The Riemann-Liouville process $R^H$ is an $H$-self-similar centred Gaussian process with a non-negative covariance function. Thus, the persistence probability of $R^H$ decays asymptotically polynomially with some persistence exponent $\theta_R: (0,\infty) \to (0,\infty)$ depending on $H$. However, similarly to $I^H,$ the exact value is unknown except for the Brownian cases $\theta_R(1/2)=1/2$ (Brownian motion) and $\theta_R(3/2)=1/4$ (integrated Brownian motion). In \cite{AurzadaDereich2013}, it was shown that $\theta_R$ is non-increasing, while in \cite{AurzadaKilian2020}, it was proven that $\theta_R$ is continuous, tends to $\infty$ and is in the range $H^{-1}$ to $H^{-2}$ for $H \to 0$. 

The autocovariance function of the Lamperti transform of $R^H$ can be found in \cite[Equation~(12)]{Lim2015} and one assures oneself that this is again a continuous and integrable function. Thus, Theorem \ref{thm:thetheorem} yields the following corollary. 
\begin{corollary}
For $0<K<H$ and independent Brownian motions $W^{(1)}$ and $W^{(2)},$ define $R_t^H:=\int_0^t (t-s)^{H-\frac {1} {2}} \dd W_s^{(1)}$ and $R_t^K:=\int_0^t (t-s)^{K-\frac {1} {2}} \dd W_s^{(2)}$. Let $a,b \in \R$ with $ab \neq 0$ and $\theta_R: (0,\infty) \to (0,\infty)$ denote the persistence exponent of the Riemann-Liouville process depending on the Hurst parameter. Then
\begin{equation*}
\Prob{\left( \sup_{t\in[0,T]} a R^H_t + b R^K_t \leq 1 \right)} = T^{-\theta_R(H) +o(1)}, \qquad T \to \infty.
\end{equation*}
\end{corollary}
Again, in the case $ab>0,$ the same result also holds for the completely correlated mixed Riemann-Liouville process, where $R^H$ and $R^K$ are generated by the same Brownian motion (instead of two independent Brownian motions), as the covariance function of the mixed process gets additional summands which are non-negative.

\section{Proof of the main result}\label{sec:proof}
In this section, we give the proof of Theorem \ref{thm:thetheorem}. The main idea is as follows. We restrict the interval $[0,T]$ of persistence to an interval $[a(T),T],$ where $a(T)$ has to be small enough such that the asymptotic order of the persistence probability does not change and large enough such that we are able to control the range of the process $Y^K$ on the interval $[a(T),T]$. It turns out that $a(T):=(\log T)^p$ for $p$ large enough is a suitable choice. The following lemma shows that the probability that $Y_t^K$ exceeds $t^\gamma$ for $\gamma>K$ on the interval $[a(T),T]$ is of neglectable order. 
\begin{lemma}\label{lem:fbmlarge}
Let $Y^K$ be as in Theorem \ref{thm:thetheorem}, $\theta \ge 0, \ \gamma > K$ and $\delta > 0$. Then there is a $p \ge e^2$ such that for $T$ large enough, it holds
\begin{equation*}
\Prob{\left( \exists t\in [ (\log T)^p,T] \, : \, \left|Y^K_t\right| >  t^\gamma\right)} \leq T^{-\theta - \delta}.
\end{equation*}
\end{lemma}
\begin{proof}
We estimate
\begin{align}
\Prob{\left( \exists t\in [ (\log T)^p,T] \, : \, \left|Y^K_t\right| >  t^\gamma\right)} &\le \sum_{s=\lfloor (\log T)^p \rfloor}^{\lfloor T \rfloor} \Prob{\left( \exists t \in [s,s+1] \, : \, \left|Y^K_t\right| >  t^\gamma\right)} \notag
\\&\le \sum_{s=\lfloor (\log T)^p \rfloor}^{\lfloor T \rfloor} \Prob{\left( \sup_{t \in [s,s+1]} \left|Y^K_t\right| >  s^\gamma\right)}. \label{eq:sum}
\end{align}
For $s = \lfloor (\log T)^p \rfloor, \dots, \lfloor T \rfloor$ and $\sigma_K^2:=\V[Y_1^K] \vee \sup_{t \in [1,2]} \V[Y_t^K-Y_1^K],$ we may further estimate
\begin{align}
&\Prob{\left( \sup_{t \in [s,s+1]} \left|Y^K_t\right| >  s^\gamma\right)} \notag
\\&\le \Prob{\left( \left|Y^K_s\right| >  \frac {s^\gamma} {2}\right)} + \Prob{\left( \sup_{t \in [s,s+1]} \left|Y^K_t\right| >  s^\gamma, \ \left|Y^K_s\right| \le  \frac {s^\gamma} {2}\right)} \notag
\\&\le \Prob{\left(|\mathcal{N}(0,1)| > \frac {s^{\gamma - K}} {2 \,\sigma_K}\right)} + \Prob{\left( \sup_{t \in [s,s+1]} \left|Y^K_t\right| - \left|Y^K_s\right| >  \frac {s^\gamma} {2}\right)} \notag
\\&\le c_1 \,e^{-s^{2 (\gamma-K)}/(8 \,\sigma_K^2)} + \Prob{\left( \sup_{t \in [s,s+1]} \left|Y^K_t - Y^K_s\right| >  \frac {s^\gamma} {2}\right)} \notag
\\&= c_1 \,e^{-s^{2 (\gamma-K)}/(8 \,\sigma_K^2)} + \Prob{\left( \sup_{t' \in [1,1+s^{-1}]} \left|Y^K_{t'} - Y^K_1\right| >  \frac {s^{\gamma-K}} {2}\right)} \label{eq:supremum}
\end{align}
for some constant $c_1 > 0$ and $T$ large enough, where we used self-similarity of $Y^K$ in the second, the reverse triangle inequality in the third, and again self-similarity in the fourth step. 

Now, we estimate the probability in \eqref{eq:supremum} as follows:
\begin{align}
&\Prob{\left( \sup_{t' \in [1,1+s^{-1}]} \left|Y^K_{t'} - Y^K_1\right| >  \frac {s^{\gamma-K}} {2}\right)} \le \Prob{\left( \sup_{t' \in [1,2]} \left|Y^K_{t'} - Y^K_1\right| >  \frac {s^{\gamma-K}} {2}\right)} \notag
\\&\le \Prob{\left( \sup_{t' \in [1,2]} \left(Y^K_{t'} - Y^K_1\right) >  \frac {s^{\gamma-K}} {2}\right)} + \Prob{\left( \sup_{t' \in [1,2]} \left(Y^K_1 - Y^K_{t'}\right) >  \frac {s^{\gamma-K}} {2}\right)} \notag
\\&=2 \,\Prob{\left( \sup_{t \in [1,2]} \left(Y^K_t - Y^K_1\right) >  \frac {s^{\gamma-K}} {2}\right)}. \label{eq:modulus}
\end{align}
The last probability is a probability of large deviation of a bounded Gaussian random function and can therefore be estimated by the tail of a one-dimensional Gaussian distribution. 

More precisely, by e.g.~\cite[Theorem~12.1]{Lifshits1995}, there exist constants $c_2>0$ and $d \in \R$ such that
\begin{equation*}
\Prob{\left( \sup_{t \in [1,2]} \left(Y^K_t - Y^K_1\right) >  \frac {s^{\gamma-K}} {2}\right)} \le c_2 \,e^{s^{\gamma-K}/2 - \left(s^{\gamma-K}/2 + d\right)^2/(2 \sigma_K^2)}.
\end{equation*}
Together with \eqref{eq:modulus} and \eqref{eq:supremum}, this yields for $s = \lfloor (\log T)^p \rfloor, \dots, \lfloor T \rfloor$:
\begin{align*}
\Prob{\left( \sup_{t \in [s,s+1]} \left|Y^K_t\right| >  s^\gamma\right)} &\le e^{-s^{2 (\gamma-K)}/(8 \,\sigma_K^2)+c_3 s^{\gamma-K}} 
\\&\le e^{-(\log T)^{2 (\gamma-K) p}/(8 \,\sigma_K^2)+c_0 (\log T)^{(\gamma-K) p}}
\end{align*}
for constants $c_3, c_0 > 0$. Combining this with \eqref{eq:sum}, we get
\begin{align*}
&\Prob{\left( \exists t\in [ (\log T)^p,T] \, : \, \left|Y^K_t\right| >  t^\gamma\right)} 
\\&\le (T-(\log T)^p+2) \,e^{-(\log T)^{2 (\gamma-K) p}/(8 \,\sigma_K^2)+c_0 (\log T)^{(\gamma-K) p}}.
\end{align*}
Taking e.g.~$p=\max\{1/(\gamma-K), e^2\},$ the right-hand-side decays faster than any polynomial, which shows the assertion. 
\end{proof}
Thus, we can estimate the persistence probability of $X^H+Y^K$ on $[a(T),T]$ by the persistence probability of $X^H$ shifted by $t \mapsto t^\gamma$ on $[a(T),T]$. Shifting a Gaussian process by a deterministic function does not change the asymptotic order of the persistence probability if the function belongs to the reproducing kernel Hilbert space (RKHS) of the process, see \cite[Proposition~1.6]{AurzadaDereich2013}. Thus, we have to estimate $t \mapsto t^{\gamma}$ on $[a(T),T]$ by a function in the RKHS of $X^H,$ which works for $\gamma<H$. The following lemma provides such a function in the RKHS of $X^H$. 
\begin{lemma}\label{lem:RKHS}
Let $X^H$ be as in Theorem \ref{thm:thetheorem}. Then, for every $\alpha \in (0,1/2),$ there exists a function $h$ in the RKHS of $X^H$ satisfying $h(t) \sim c \,t^H (\log t)^{\alpha - 1}$ for $t \to \infty$ and some $c > 0$ as well as $h(t) \ge 1$ for $t \ge 1$. 
\end{lemma}
\begin{proof}
Let $\alpha \in (0,1/2)$ and $Z^H(\tau):=e^{-\tau H} X_{e^\tau}^H, \ \tau \in \R,$ be the Lamperti transform of $X^H$. We will show below that there is a function $\tilde{h}$ in the RKHS of $Z^H$ satisfying 
\begin{equation}\label{eq:htilde}
\tilde{h}(\tau) \sim c_0 \tau^{\alpha-1} \text{ for } \tau \to \infty \text{ and some } c_0>0 \text{ as well as } \tilde{h}(\tau) > 0 \text{ for } \tau \ge 0. 
\end{equation}
Then, by the definition of the RKHS, there exists a random variable $\xi$ in the $L^2$-closure of $\Span\{Z^H_\tau\colon \tau \in \R\}=\Span\{X^H_t\colon t > 0\}$ such that $\tilde{h}(\tau)=\E[\xi Z^H_\tau]$. Plugging in the definition of $Z^H,$ this gives $e^{\tau H} \tilde{h}(\tau)=\E[\xi X^H_{e^\tau}]$ for $\tau \in \R$ and $h_0(t):=t^H \tilde{h}(\log t)=\E[\xi X^H_t]$ for $t > 0$. So, by definition, the function $h_0$ is an element of the RKHS of $X^H,$ a continuous function (by the continuity of the covariance function) and satisfies $h_0(t) \sim c_0 \,t^H (\log t)^{\alpha-1}$ for $t \to \infty$ as well as $h_0(t) > 0$ for all $t \ge 1$. In particular, we have $h_0(t) \to \infty$ for $t \to \infty$. Thus, there exists $t_0 > 1$ such that $h_0(t) \ge 1$ for $t \ge t_0$. Setting $h:=h_0/(\min_{t \in [1,t_0]} h_0(t) \wedge 1)$ yields the assertion for $c:=c_0/(\min_{t \in [1,t_0]} h_0(t) \wedge 1)$. 

So let us show \eqref{eq:htilde}. Due to the continuity of the covariance function of $Z^H,$ by Bochner's theorem, the spectral measure of the stationary process $Z^H$ exists, i.e.~the finite measure $\mu$ that satisfies
\begin{equation}\label{eq:spectral}
r(\tau):=\E[Z^H_0 Z^H_\tau]=\int_\R e^{i \tau x} \dd \mu(x), \qquad \tau \in \R.
\end{equation}
Recall that in this case, a function $\tilde{h}$ is in the RKHS of $Z^H$ if and only if there is a function $\varphi \in L^2(\mu)$ such that $\tilde{h}(\tau)=\int_\R \varphi(x) e^{-i \tau x} \dd \mu(x),$ see e.g.~\cite[Comment~2.2.2(c)]{AshGardner1975}. 

Note that \eqref{eq:spectral} states that $r$ is the characteristic function of the (finite) measure $\mu$. As $r$ is integrable by assumption, the inversion theorem for characteristic functions gives $\dd \mu(x)=p(x) \dd x$ on $\R$ with the density
\begin{equation*}
p(x)=\frac {1} {2 \pi} \int_\R e^{-i \tau x} r(\tau) \dd \tau \to \frac {1} {2 \pi} \int_\R r(\tau) \dd \tau=:2 \,c_1 \in (0,\infty), \qquad x \to 0,
\end{equation*}
where $c_1>0$ is due to the non-negativity of $r$. In particular, there exists $x_0 > 0$ such that $p(x) \ge c_1$ for $|x|<x_0$. 

Now, similary to the proof of \cite[Proposition~5]{AurzadaBuck2018}, we will first construct a function $\tilde{h}_1$ in the RKHS of $Z^H$ with the desired asymptotic behaviour, which unfortunately may attain non-positive values up to some $\tau_0>0$. Afterwards, we will show the existence of another function $\tilde{h}_2$ in the RKHS of $Z^H$ which is non-negative on $[0,\infty),$ even positive on $[0,\tau_0]$ and decays faster than $\tilde{h}_1$. Then, the function $\tilde{h}:=\tilde{h}_1 + 2 \max_{\tau \in [0,\tau_0]} |\tilde{h}_1(\tau)| / \min_{\tau \in [0,\tau_0]} \tilde{h}_2(\tau) \cdot \tilde{h}_2$ satisfies condition \eqref{eq:htilde}. 

{\it Construction of $\tilde{h}_1$:} 
We set $\varphi_1(x):=\mathbbm{1}_{|x|<x_0} \,|x|^{-\alpha}/p(x)$. Then, we have
\begin{equation*}
\int_\R \varphi_1^2(x) \dd \mu(x)=2 \int_0^{x_0} x^{-2 \alpha}/p(x) \dd x \le 2 \,c_1^{-1} \int_0^{x_0} x^{- 2 \alpha} \dd x < \infty
\end{equation*}
as $\alpha < 1/2$. Thus $\varphi_1 \in L^2(\mu)$. For the $\tilde{h}_1$ corresponding to $\varphi_1,$ we get
\begin{align}\label{eq:asymptotics}
\tilde{h}_1(\tau)&=\int_\R \varphi_1(x) e^{-i \tau x} \dd \mu(x)=\int_\R \varphi_1(x) \cos(\tau x) \dd \mu(x) \notag
\\&=2 \int_0^{x_0} x^{-\alpha} \cos(\tau x) \dd x = 2 \tau^{\alpha-1} \int_0^{\tau x_0} y^{-\alpha} \cos(y) \dd y \sim c_0 \tau^{\alpha-1}
\end{align}
for $c_0:=2 \int_0^\infty y^{-\alpha} \cos(y) \dd y$ and $\tau \to \infty$. Note that due to $\alpha<1,$ the fact that $\cdot^{-\alpha}$ is decreasing and fulfills $\lim_{x \to \infty} x^{-\alpha}=0,$ the fact that the integrals of $\cos(\cdot)$ over any interval are uniformly bounded and Dirichlet's test, the integral in the definition of $c_0$ exists and is positive. Further note that this fails for $\alpha \le 0$. 

{\it Construction of $\tilde{h}_2$:}
Due to the asymptotic behaviour of $\tilde{h}_1,$ there exists $\tau_0 > \pi/x_0$ such that $\tilde{h}_1(\tau) > 0$ for $\tau \ge \tau_0$. Let $g\colon \R \to [0,\infty)$ be a smooth even function with $g(x) > 0$ for $|x| < \pi/(2 \tau_0)$ and $g(x)=0$ otherwise. 

We set $f:=g \ast g$ and $\varphi_2(x):=\mathbbm{1}_{|x|<\pi/\tau_0} \,f(x)/p(x)$. Then $\varphi_2 \in L^2(\mu)$ as
\begin{equation*}
\int_\R \varphi_2^2(x) \dd \mu(x)=\int_{-\pi/\tau_0}^{\pi/\tau_0} f^2(x)/p(x) \dd x \le \frac {2 \pi \max_{x \in [-\pi/\tau_0,\pi/\tau_0]} f^2(x)} {c_1 \tau_0}  < \infty,
\end{equation*}
where we used that $\pi/\tau_0 < x_0$. Note that by definition of $f$ and $g,$ we have $f(x)=0$ for $|x| \ge \pi/\tau_0$. Thus, the $\tilde{h}_2$ corresponding to $\varphi_2$ fulfills
\begin{align*}
\tilde{h}_2(\tau)&=\int_\R \varphi_2(x) e^{-i \tau x} \dd \mu(x)=\int_\R f(x) e^{-i \tau x} \dd x
\\&=\left(\int_\R g(x) e^{-i \tau x} \dd x\right)^2=\left(\int_\R g(x) \cos(\tau x) \dd x\right)^2 \begin{cases}
>0, &\text{if } |\tau| \le \tau_0,
\\\ge 0, &\text{otherwise,}
\end{cases}
\end{align*}
where we used in the second line that the Fourier transform of a convolution is given by the product of the Fourier transforms of the convoluted functions as well as that $g$ vanishes outside of $(-\pi/(2 \tau_0), \pi/(2 \tau_0))$ by definition. Furthermore, by integration by parts, we have
\begin{align*}
\tilde{h}_2(\tau)=\int_\R f(x) e^{-i \tau x} \dd x&=\frac {1} {(i \tau)^2} \int_\R f''(x) e^{-i \tau x} \dd x 
\\&\le \frac {2 \pi \max_{x \in [-\pi/\tau_0,\pi/\tau_0]} |f''(x)|} {\tau_0} \cdot \tau^{-2}. \qedhere
\end{align*}
\end{proof}

\begin{proof}[Proof of Theorem~\ref{thm:thetheorem}]
First observe that the persistence probability of $X^H$ behaves as in \eqref{eq:pppolynomial} with some persistence exponent $\theta_X \in [0,\infty)$. Indeed, as $X^H$ is a self-similar centred Gaussian process with a non-negative covariance function by assumption, its Lamperti transform $Z^H(\tau):=e^{-\tau H} X_{e^\tau}^H, \ \tau \in \R,$ is a stationary centred Gaussian process with a non-negative correlation function. By subadditivity and Slepian's lemma, this yields
\begin{equation*}
\Prob{\left( \sup_{t\in[1,T]} X^H_t \leq 0 \right)} = \Prob{\left( \sup_{\tau\in[0,\log T]} Z^H_\tau \leq 0 \right)} = T^{-\theta_X + o(1)}, \qquad T \to \infty,
\end{equation*}
for some $\theta_X \in [0,\infty)$. Consider the function $h$ in the RKHS of $X^H$ given by Lemma \ref{lem:RKHS} for $\alpha:=1/4$. On the one hand, we know by Slepian's lemma that
\begin{align*}
\Prob{\left( \sup_{t\in[0,T]} X^H_t \leq 1 \right)} &\ge \Prob{\left( \sup_{t\in[0,1]} X^H_t \leq 1 \right)} \cdot \Prob{\left( \sup_{t\in[1,T]} X^H_t \leq 1 \right)}
\\&\ge \Prob{\left( \sup_{t\in[0,1]} X^H_t \leq 1 \right)} \cdot \Prob{\left( \sup_{t\in[1,T]} X^H_t \leq 0 \right)},
\end{align*}
while on the other hand, the fact that $h(t) \ge 1$ for $t \ge 1$ together with \cite[Proposition~1.6]{AurzadaDereich2013} yields that
\begin{align*}
\Prob{\left( \sup_{t\in[0,T]} X^H_t \leq 1 \right)} \le \Prob{\left( \sup_{t\in[1,T]} X^H_t \leq 1 \right)} &\le \Prob{\left( \sup_{t\in[1,T]} X^H_t+h(t) \leq 1 \right)} \,T^{o(1)}
\\&\le \Prob{\left( \sup_{t\in[1,T]} X^H_t \leq 0 \right)} \,T^{o(1)},
\end{align*}
showing that also $\Prob{\left( \sup_{t\in[0,T]} X^H_t \leq 1 \right)} = T^{-\theta_X+o(1)}$ for $T \to \infty$. 

Now, note that, by \cite[Proposition~1.6]{AurzadaDereich2013}, we have
\begin{equation}\label{eq:dereich}
\frac {\Prob{\left(\sup_{t \in [(\log T)^p,T]} X_t^H \pm h(t) \le 1\right)}} {\Prob{\left(\sup_{t \in [(\log T)^p,T]} X_t^H \le 1\right)}} = T^{o(1)}, \qquad T \to \infty.
\end{equation}
Fix $\gamma$ with $K<\gamma<H$. 

{\it Upper bound:} 
\begin{align*}
&\Prob{\left( \sup_{t\in[0,T]} X^H_t + Y^K_t \leq 1 \right)} 
\\
&\leq \Prob{\left( \exists t\in [ (\log T)^p,T] \, : \, \left|Y^K_t\right| > h(t)\right)} + \Prob{\left( \sup_{t\in[(\log T)^p,T]} X^H_t -h(t) \leq 1 \right)} 
\\
&\leq \Prob{\left( \exists t\in [ (\log T)^p,T] \, : \, \left|Y^K_t\right| > t^\gamma\right)} + \Prob{\left( \sup_{t\in[(\log T)^p,T]} X^H_t \leq 1 \right)} \,T^{o(1)}
\\
&\leq T^{-\theta_X-\delta} + \frac{ \Prob{\left( \sup_{t\in[0,T]} X^H_t \leq 1 \right)} }{\Prob{\left( \sup_{t\in[0,(\log T)^p]} X^H_t \leq 1 \right)}}  \,T^{o(1)}
\\
&\leq T^{-\theta_X-\delta} + T^{-\theta_X+o(1)} (\log T)^{p \,\theta_X+o(1)} \, T^{o(1)} 
\\
&\leq T^{-\theta_X+o(1)}
\end{align*}
for $T$ large enough and $\delta >0,$ where $p$ is chosen according to Lemma \ref{lem:fbmlarge} for $\theta:=\theta_X$. Here, the second inequality uses \eqref{eq:dereich} and the property of $h$ that $h(t) \sim c \,t^H (\log t)^{-3/4} > t^\gamma$ for $t$ large enough, while the third inequality is Lemma \ref{lem:fbmlarge} together with Slepian's lemma. 

{\it Lower bound:} 
The opposite reasoning gives
\begin{align*}
T^{-\theta_X+o(1)} &= \Prob{\left( \sup_{t\in[0,T]} X^H_t \leq 1 \right)} 
\\
&\le \Prob{\left( \sup_{t\in[(\log T)^p,T]} X^H_t + h(t) \leq 1 \right)} \,T^{o(1)}
\\
&\le \Prob{\left( \exists t\in [ (\log T)^p,T] \, : \, \left|Y^K_t\right| > h(t)\right)} \,T^{o(1)}
\\
&\qquad + \Prob{\left( \sup_{t\in[(\log T)^p,T]} X^H_t + Y^K_t \leq 1 \right)} \,T^{o(1)}
\\
&\leq \Prob{\left( \exists t\in [ (\log T)^p,T] \, : \, \left|Y^K_t\right| > t^\gamma\right)} \,T^{o(1)}
\\
&\qquad + \Prob{\left( \sup_{t\in[(\log T)^p,T]} X^H_t + Y^K_t \leq 1 \right)} \,T^{o(1)}
\\
&\le T^{-\theta_X-\delta} + \frac {\Prob{\left( \sup_{t\in[0,T]} X^H_t + Y^K_t \leq 1 \right)}} {\Prob{\left( \sup_{t\in[0,(\log T)^p]} X^H_t + Y^K_t \leq 1 \right)}} \,T^{o(1)},
\end{align*}
where we used the definition of $\theta_X$ in the first, \eqref{eq:dereich} in the second and Lemma \ref{lem:fbmlarge} as well as Slepian's lemma in the fifth step. Precisely here, we use the assumption of non-negative covariances of $X^H+Y^K$. So we have
\begin{equation}\label{eq:inducstart-1}
\Prob{\left( \sup_{t\in[0,T]} X^H_t + Y^K_t \leq 1 \right)} \ge T^{-\theta_X+o(1)} \,\Prob{\left( \sup_{t\in[0,(\log T)^p]} X^H_t + Y^K_t \leq 1 \right)}.
\end{equation}
We then further estimate the right-hand side of \eqref{eq:inducstart-1} by replacing $T$ in \eqref{eq:inducstart-1} by $(\log T)^p$ and get
\begin{align}
&\Prob{\left( \sup_{t\in[0,T]} X^H_t + Y^K_t \leq 1 \right)}  \notag
\\&\ge T^{-\theta_X+o(1)} (\log T)^{-p \,\theta_X+o(1)} \,\Prob{\left( \sup_{t\in[0,(p \log \log T)^p]} X^H_t + Y^K_t \leq 1 \right)}. \label{eq:inducstart}
\end{align}
We set $f_0(T):=\log \log T$ and $f_N(T):=\log p + \log f_{N-1}(T)$ for $N \ge 1$. Using \eqref{eq:inducstart-1} iteratively then gives
\begin{align}
&\Prob{\left( \sup_{t\in[0,T]} X^H_t + Y^K_t \leq 1 \right)} \notag
\\&\ge T^{-\theta_X+o(1)} (\log T)^{(-p \,\theta_X+o(1)) (N+1)} \,\Prob{\left( \sup_{t\in[0,(p f_N(T))^p]} X^H_t + Y^K_t \leq 1 \right)} \notag
\\&=T^{-\theta_X+o(1)+\frac {(\log \log T) (N+1)} {\log T} (-p \,\theta_X+o(1))} \,\Prob{\left( \sup_{t\in[0,(p f_N(T))^p]} X^H_t + Y^K_t \leq 1 \right)}, \label{eq:Ntimes}
\end{align}
for $N \in \N$. This can be seen by induction: The induction base is \eqref{eq:inducstart}, while for the induction step, one has to note that
\begin{equation*}
\left(\log\!{\left((p f_{N-1}(T))^p\right)}\right)^p = \left(p (\log p + \log f_{N-1}(T))\right)^p = \left(p f_N(T)\right)^p.
\end{equation*}
Now we consider the function 
\begin{equation*}
\varphi_p(x):=\log p + \log x, \qquad x \in [2,\infty).
\end{equation*}
This is a contraction with Lipschitz constant $1/2$. The Lipschitz constant can be computed by the fact that $\varphi_p'(x)=1/x \le 1/2$ for $x \ge 2,$ while the self-map property of $\varphi_p$ is deduced from the fact that $\log p \ge 2$ holds by Lemma \ref{lem:fbmlarge}. Thus, the Banach fixed-point theorem yields a unique fixed-point $a_p \ge 2$ of $\varphi_p,$ which does not depend on $T$. Further, as $f_N(T)=\varphi_p(f_{N-1}(T)),$ we can estimate
\begin{align*}
\left|f_N(T)-a_p\right| &\le \frac {2^{-N}} {1-\frac {1} {2}} \left|f_1(T)-f_0(T)\right| 
\\&= 2^{1-N} \left|\log p + \log \log \log T - \log \log T\right| \le 2^{1-N} \cdot 3 \log \log T
\end{align*}
for $N \in \N$ and $T$ large enough, see e.g.~\cite[Theorem~1.1(iii)]{Agarwal2018}. For $N_T:=\lceil (\log \log \log T + \log 6)/\log 2 \rceil,$ this implies
\begin{equation*}
\left|f_{N_T}(T)-a_p\right| \le 2^{1-N_T} \cdot 3 \log \log T \le 1.
\end{equation*}
Considering \eqref{eq:Ntimes} for $N:=N_T$ consequently yields
\begin{align*}
\Prob{\left( \sup_{t\in[0,T]} X^H_t + Y^K_t \leq 1 \right)} &\ge T^{-\theta_X+o(1)} \,\Prob{\left( \sup_{t\in[0,(p \,f_{N_T}(T))^p]} X^H_t + Y^K_t \leq 1 \right)}
\\&\ge T^{-\theta_X+o(1)} \,\Prob{\left( \sup_{t\in[0,(p \,(1+a_p))^p]} X^H_t + Y^K_t \leq 1 \right)}
\\&=T^{-\theta_X+o(1)},
\end{align*}
which finishes the proof. 
\end{proof}

\noindent\textbf{Acknowledgement.} This work was supported by Deutsche
Forschungsgemeinschaft (DFG grant AU370/5). We thank Dominic Schickentanz for his useful comments improving the first draft of this paper. 

\nocite{*}
\bibliographystyle{plain}

\end{document}